\documentclass[12pt,a4paper]{amsart}
\usepackage{
    amsmath,  amsfonts, amssymb,  amsthm,   amscd,
    gensymb,  graphicx, comment,  etoolbox, url,
    booktabs, stackrel, mathtools,enumitem, mathdots    
}
\usepackage[usenames,dvipsnames]{xcolor}
\usepackage[utf8]{inputenc}
\usepackage{lmodern, microtype, fullpage, wrapfig,textcomp,mathrsfs}
\usepackage[colorlinks=true, linkcolor=blue, citecolor=blue, urlcolor=blue, breaklinks=true]{hyperref}
\usepackage[capitalise]{cleveref}
\usepackage{todonotes}
\newtheorem{theorem}{Theorem}[section]

\newtheorem{corollary}[theorem] {Corollary}
\newtheorem{definition}[theorem]{Definition}
\newtheorem{example}[theorem]{Example}
\newtheorem{lemma}[theorem]{Lemma}

\newtheorem{proposition}[theorem]{Proposition}
\newtheorem{remark}[theorem]{Remark}
\usepackage{tikz}
\usepackage{float}
\usepackage{tabularx}
\usepackage{kantlipsum}

\newcommand\R{\mathbb{R}}

\newcommand{\A}{\mathcal{A}}
\newcommand{\ipa}{\mathrm{L}}
\newcommand{\seqnum}[1]{\href{https://oeis.org/#1}{\rm \underline{#1}}}

\begin{document}

\title{Counting regions of the boxed threshold arrangement}
\author{Priyavrat Deshpande}
\address{Chennai Mathematical Institute}
\email{pdeshpande@cmi.ac.in}
\author{Krishna Menon}
\address{Chennai Mathematical Institute}
\email{krishnamenon@cmi.ac.in}
\author{Anurag Singh}
\address{Chennai Mathematical Institute}
\email{anuragsingh@cmi.ac.in}
\thanks{PD and AS are partially supported by a grant from the Infosys Foundation}

\begin{abstract}
In this paper we consider the hyperplane arrangement in $\R^n$ whose hyperplanes are $\{x_i + x_j = 1\mid 1\leq i < j\leq n\}\cup \{x_i=0,1\mid 1\leq i\leq n\}$. 
We call it the \emph{boxed threshold arrangement} since we show that the bounded regions of this arrangement are contained in an $n$-cube and are in one-to-one correspondence with 
the labeled threshold graphs on $n$ vertices. 
The problem of counting regions of this arrangement was studied earlier by Joungmin Song. 
He determined the characteristic polynomial of this arrangement by relating its coefficients to the count of certain graphs. 
Here, we provide bijective arguments to determine the number of regions. 
In particular, we construct certain signed partitions of the set $\{-n,\dots, n\}\setminus\{0\}$ and also construct colored threshold graphs on $n$ vertices and show that both these objects are in bijection with the regions of the boxed threshold arrangement. 
We independently count these objects and provide a closed form formula for the number of regions. 
\end{abstract}

\maketitle

\section{Introduction}
A \emph{hyperplane arrangement} $\A$ is a finite collection of affine hyperplanes (i.e., codimension $1$ subspaces and their translates) in $\R^n$. 
Without loss of generality we assume that arrangements we consider are \emph{essential}, i.e., the subspace spanned by the normal vectors is the ambient vector space.   
Removing all the hyperplanes of $\A$ leaves $\R^n$ disconnected; counting the number of connected components using diverse combinatorial methods is an active area of research. 
A \emph{flat} of $\A$ is a nonempty intersection of some of the hyperplanes in $\A$; the ambient vector space is a flat since it is an intersection of no hyperplanes. 
Flats are naturally ordered by reverse set inclusion; the resulting poset is called the \emph{intersection poset} and is denoted by $\ipa(\A)$. A \emph{region} of $\A$ is a connected component of $\R^n\setminus \bigcup \A$. 
The \emph{characteristic polynomial} of $\A$ is defined as 
\[\chi_{\A}(t) := \sum_{x\in\ipa(\A)} \mu(\hat{0},x)\, t^{\dim(x)}\]
where $\mu$ is the M\"obius function of the intersection poset and $\hat{0}$ corresponds to the flat $\R^n$. 
The characteristic polynomial is a fundamental combinatorial and topological invariant of the arrangement and plays a significant role throughout the theory of hyperplane arrangements. Among other things, the polynomial encodes enumerative information about the stratification of the space $\R^n$, induced by the arrangement.
We refer the reader to Stanley's notes \cite{stanarr07} for more on the enumerative aspects of hyperplane arrangements. 
In particular we have the following seminal result by Zaslavsky.

\begin{theorem}[\cite{zas75}]
Let $\A$ be an arrangement in $\R^n$. Then the number of regions of $\A$ is given by 
\[ r(\A) = (-1)^n \chi_{\A}(-1)\]
and the number of bounded regions is given by 
\[b(\A) = (-1)^n \chi_{\A}(1). \]
\end{theorem}

The finite field method,  developed by Athanasiadis \cite{athanas96}, converts the computation of the characteristic polynomial to a point counting problem. 
A combination of these two results allowed for the computation of the number of regions of several arrangements of interest. 

Another way to count the number of regions is to give a bijective proof. 
This approach involves finding a combinatorially defined set whose elements are in bijection with the regions of the given arrangement and are easier to count. 
For example, the regions of the braid arrangement (whose hyperplanes are given by the equations $x_i-x_j = 0$ for $1\leq i<j\leq n$) correspond to the $n!$ permutations of order $n$. 
The regions of the Shi arrangement (whose hyperplanes are given by the equations $x_i-x_j = 0, 1$ for $1\leq i< j\leq n$) are in bijection with the parking functions on $[n]$, hence the number of regions is $(n+1)^{n-1}$. 
Refer to Stanley's notes \cite[Lecture 5]{stanarr07} for details. 

In the present article we consider the following hyperplane arrangement 
\[\mathcal{BT}_n := \{X_i + X_j =1 \mid 1\leq i< j\leq n\}\cup \{X_i = 0,1\mid 1\leq i\leq n \}. \]
The problem of counting regions of this arrangement was recently solved by Song in a series of papers. 
His first approach \cite{songcolor17, songenum17} involved relating the coefficients of the characteristic polynomial to the generating functions for the number of certain graphs. 
These generating functions \cite[Theorem 1]{songenum17} have quite a complicated expression and consequently make it difficult to determine a compact form for the characteristic polynomial. 
Later Song \cite{songchar18} used the finite field method to get a slightly simpler expression for the characteristic polynomial. 

The main aim of this paper is to provide bijective proofs for the number of regions of this arrangement. 
In particular, in \cref{section3} we construct certain (signed) ordered partitions of the set $\{-n, -(n-1),\dots,n-1, n\}\setminus \{ 0\}$ and show that they are in bijection with the regions of $\mathcal{BT}_n$. 
We also show how to count these partitions. 
In \cref{section4} we spell out a recipe to color the vertices of a labeled threshold graph on $n$ vertices such that the number of such colored threshold graphs is equal to the regions of $\mathcal{BT}_n$. 
However, we begin the article by establishing a simpler form for the characteristic polynomial in \cref{sectio2}. 
There we also justify the term ``\emph{boxed threshold}".

\section{The characteristic polynomial}\label{sectio2}
We first translate the hyperplanes in $\mathcal{BT}_n$ in order to obtain a combinatorially isomorphic arrangement with the same characteristic polynomial.
Putting $X_i=x_i+\frac{1}{2}$ for every $i$ we get
\begin{equation}
\{x_i + x_j = 0\mid 1\leq i < j\leq n\}\cup \{x_i=-\frac{1}{2},\frac{1}{2}\mid 1\leq i\leq n\}.\label{changed equation}
\end{equation}
We will stick to the notation $\mathcal{BT}_n$ to denote the above arrangement.
The notation $[k]$ is used to denote the set $\{1,\dots,k \}$.
We begin by stating a generalization of the finite field method, given by Athanasiadis \cite{athanasiadis1999extended}, in our context. 

\begin{theorem}[{\cite[Theorem 2.1]{athanasiadis1999extended}}]\label{ffm}
If $\mathcal{A}$ is a sub-arrangement of the type $C$ arrangement in $\R^n$, that is, a sub-arrangement of $\{x_i\pm x_j=0 \mid 1 \leq i < j \leq n\} \cup \{x_i = 0 \mid i \in [n]\}$, there exists an integer $k$ such that for all odd integers $q$ greater than $k$,
    \begin{equation*}
        \chi_\mathcal{A}(q)=\#(\mathbb{Z}_q^n \setminus V_{\mathcal{A}})
    \end{equation*}
    where $V_{\mathcal{A}}$ is the union of hyperplanes in $\mathbb{Z}_q^n$ obtained by reducing $\mathcal{A}$ mod q.
\end{theorem}

We use this result to prove the following relationship between the characteristic polynomial of sub-arrangements of the type $C$ arrangement and that of their ``boxed" versions.

\begin{proposition}
Let $\mathcal{A}$ be an arrangement in $\mathbb{R}^n$ that is a sub-arrangement of the type $C$ arrangement and  let $\mathcal{BA}=\mathcal{A} \cup \{x_i = -\frac{1}{2}, \frac{1}{2} \mid i \in [n]\}$. 
Then 
\begin{equation*}
    \chi_{\mathcal{BA}}(t)=\chi_{\mathcal{A}}(t-2).
\end{equation*}
\end{proposition}

\begin{proof}
Let $q$ be any large odd number.
Set $D_q^n := \{(a_1, \dots, a_n) \in \mathbb{Z}_q^n \mid a_i \neq \pm \frac{q-1}{2}\}$.
Define a bijection $f : \mathbb{Z}_{q-2} \rightarrow  \mathbb{Z}_q \setminus \{\frac{q-1}{2},-\frac{q-1}{2}\}$ as
\begin{equation*}
    f(i) = i \quad \text{for }i \in [-\frac{q-3}{2},\frac{q-3}{2}].
\end{equation*}
It is clear that for any $a,b \in \mathbb{Z}_{q-2}$
\begin{enumerate}
    \item $a+b = 0 \Leftrightarrow f(a)+f(b) = 0$,
    \item $a-b = 0 \Leftrightarrow f(a)-f(b) = 0$, and
    \item $a = 0 \Leftrightarrow f(a) = 0$.
\end{enumerate}
Using $f$, we can define a bijection $F : \mathbb{Z}_{q-2}^n \rightarrow D_q^n$ as
\begin{equation*}
    F(a_1,\dots,a_n) = (f(a_1),\dots,f(a_n))\quad \text{for }(a_1,\dots,a_n) \in \mathbb{Z}_{q-2}^n.
\end{equation*}
By the properties of $f$, we can see that $F$ induces a bijection between those tuples in $\mathbb{Z}_{q-2}^n$ that do not satisfy the defining equation of any hyperplane in $\mathcal{A}$ and those tuples in $\mathbb{Z}_q^n$ that do not satisfy the defining equation of any hyperplane in $\mathcal{BA}$.
So, we get that for large odd numbers $q$,
\begin{equation*}
    \chi_{\mathcal{BA}}(q)=\chi_{\mathcal{A}}(q-2).
\end{equation*}
Since $\chi_{\mathcal{BA}}$ and $\chi_{\mathcal{A}}$ are polynomials, we get the required result.
\end{proof}
Let $\mathcal{T}_n$ denote the threshold arrangement in $\R^n$, i.e., 
$\mathcal{T}_n := \{x_i + x_j = 0 \mid 1\leq i < j\leq n \}$. 
The reason this is called the threshold arrangement is that its regions are in bijection with labeled threshold graphs on $n$ vertices (see \cref{section4} for details).
This is clearly a sub-arrangement of the type $C$ arrangement. 

\begin{corollary}\label{boxth corollary} The characteristic polynomials of $\mathcal{BT}_n$ and $\mathcal{T}_n$ are related as follows: 
\[\chi_{\mathcal{BT}_n}(t)=\chi_{\mathcal{T}_n}(t-2). \]
\end{corollary}

Consequently, the number of bounded regions of $\mathcal{BT}_n$ is equal to the number of regions of $\mathcal{T}_n$. 
Moreover, these bounded regions are contained in the cube (or a \emph{box}) $\left[-\frac{1}{2},\frac{1}{2} \right]^n$.
Next, we derive a closed form expression for $\chi_{\mathcal{T}_n}(t)$ using the finite field method. 

\begin{proposition}
The characteristic polynomial of the threshold arrangement $\mathcal{T}_n$ is
\begin{equation*}
    \chi_{\mathcal{T}_n}(t)=\sum_{k = 1}^n (S(n,k)+n \cdot S(n-1,k))\prod_{i=1}^k(t-(2i-1)).
\end{equation*}
Here $S(n,k)$ are the Stirling numbers of the second kind. 
\end{proposition}
\begin{proof}
Using \cref{ffm}, we see that the characteristic polynomial of $\mathcal{T}_n$ satisfies, for large odd values of $q$,
\begin{equation*}
    \chi_{\mathcal{T}_n}(q)=|\{(a_1,\dots,a_n) \in \mathbb{Z}_q^n \mid a_i+a_j \neq 0\text{ for all } 1 \leq i < j \leq n\}|.
\end{equation*}This means that we need to count the functions $f:[n] \rightarrow \mathbb{Z}_q$ such that
\begin{enumerate}
    \item $|f^{-1}(0)| \leq 1$.
    \item $f$ can take at most one value from each of the sets
    \begin{equation*}
        \{1,-1\},\{2,-2\},\dots,\{\frac{q-1}{2},-\frac{q-1}{2}\}.
    \end{equation*}
\end{enumerate}

We split the count into the two cases. If 0 is not attained by $f$, then all values must be from
\begin{equation*}
    \{1,-1\}\cup\{2,-2\}\cup\cdots\cup\{\frac{q-1}{2},-\frac{q-1}{2}\}.
\end{equation*}with at most one value attained in each set. So, there are
\begin{equation*}
    \binom{\frac{q-1}{2}}{k} \cdot 2^k \cdot k! \cdot S(n,k)
\end{equation*}ways for $f$ to attain values from exactly $k$ of these sets. This is because we have $\binom{\frac{q-1}{2}}{k} \cdot 2^k$ ways to choose the $k$ sets and which element of each set $f$ should attain and $k!S(n,k)$ ways to choose the images of the elements of $[n]$ after making this choice. So the total number of $f$ such that 0 is not attained is
\begin{equation*}
    \sum\limits_{k=1}^{n}\binom{\frac{q-1}{2}}{k} \cdot 2^k \cdot k! \cdot S(n,k).
\end{equation*}

When 0 is attained, there are $n$ ways to choose which element of $[n]$ gets mapped to 0 and using a similar logic for choosing the images of the other elements, we get that the total number of $f$ where 0 is attained is
\begin{equation*}
    n \cdot \sum\limits_{k=1}^{n-1}\binom{\frac{q-1}{2}}{k} \cdot 2^k \cdot k! \cdot S(n-1,k).
\end{equation*}

So we get that for large $q$,
\begin{align*}
    \chi_{\mathcal{T}_n}(q)&=\sum\limits_{k=1}^{n}\binom{\frac{q-1}{2}}{k} \cdot 2^k \cdot k! \cdot S(n,k) + n \cdot \sum\limits_{k=1}^{n-1}\binom{\frac{q-1}{2}}{k} \cdot 2^k \cdot k! \cdot S(n-1,k)\\
    &=\sum_{k = 1}^n (S(n,k)+n \cdot S(n-1,k))\prod_{i=1}^k(q-(2i-1)).
\end{align*}Since $\chi_{\mathcal{T}_n}$ is a polynomial, we get the required result.
\end{proof}
\begin{remark}\label{thresh remark} Note that the absolute value of the coefficient of $t^j$ in $(t-1)(t-3)\cdots (t-(2k-1))$ counts the number of signed permutations on $[k]$ with $j$ odd cycles (see \seqnum{A039757} in the OEIS \cite{OEIS}).
Let us denote that number by $a(k,j)$.
Using this we get a compact expression for the coefficient of $t^j$ in $\chi_{\mathcal{T}_n}(t)$ as
\begin{equation*}
    \sum_{k=j}^n(S(n,k)+n \cdot S(n-1,k)) a(k,j).
\end{equation*}
\end{remark}

\begin{corollary}
The coefficient of $t^j$ in $\chi_{\mathcal{BT}_n}(t)$ is
\begin{equation*}
    \sum_{k=j}^n (S(n,k)+n \cdot S(n-1,k)) b(k,j).
\end{equation*}
where $b(k,j)$ is the coefficient of $t^j$ in $(t-3)(t-5)\cdots (t-(2k+1))$. 
\end{corollary}

\begin{proof}
Using \Cref{boxth corollary} we get
\[
\chi_{\mathcal{BT}_n}(t) = \sum_{k = 1}^n (S(n,k)+n \cdot S(n-1,k))\prod_{i=1}^k(t-(2i+1)).
\]
Expanding the product gives us the required expression.
Note that $b(k,j)=-\sum\limits_{i=0}^j a(k+1,i)$ where $a(k,j)$ is defined in \Cref{thresh remark}.
\end{proof}

\begin{remark}
We can also derive an expression for the exponential generating function for the characteristic polynomial. 
Using Problem $25(c)$ in Stanley's notes \cite[Lecture 5]{stanarr07} we get
\[\sum_{n\geq 0} \chi_{\mathcal{BT}_n}(t) \frac{x^n}{n!} = (1+x)(2e^x-1)^{\frac{(t-3)}{2}}.  \]
The generating function for the number of regions is
\[\sum_{n \geq 0} r(\mathcal{BT}_n) \frac{x^n}{n!} = \frac{e^{2x}(1-x)}{(2-e^x)^2}.\]
\end{remark}

For the sake of completeness we enumerate the coefficients of the characteristic polynomial for smaller values of $n$ (see \cref{tab:my_label}). 
Song \cite{songenum17} also computed the characteristic polynomial for $n\leq 10$, however there are typos in all the expressions for $n\geq 4$, consequently the region numbers are wrong. 
The sequence of number of regions of $\mathcal{BT}_n$ can be found at the entry \seqnum{A341769} in the OEIS \cite{OEIS}. 

\begin{table}
    \addtolength{\tabcolsep}{10pt}
    \renewcommand{\arraystretch}{1.25}
    \centering
    \begin{tabularx}{\textwidth}{|c|X|c|}
        \hline
         $n$ & \hspace{4cm}$\chi_{\mathcal{BT}_n}(t)$ & {$r(\mathcal{BT}_n)$} \\
        \hline
        2 & $t^2 -5t +6$ & 12\\
        \hline
        3 & $t^3 -9t^2 +27t -27$ & 64\\
        \hline
        4 & $t^4 -14t^3 +75t^2 -181t +165$ & 436\\
        \hline
        5 & $t^5 -20t^4 +165t^3 -695t^2 +1480t -1263$ & 3624\\
        \hline
        6 & $t^6 -27t^5 +315t^4 -2010t^3 +7320t^2 -14284t +11559$ & 35516\\
        \hline
        7 & $t^7 -35t^6 +546t^5 -4865t^4 +26460t^3 -87010t^2 +158753t -122874$ & 400544\\
        \hline
        8 & $t^8 -44t^7 +882t^6 -10402t^5 +78155t^4 -379666t^3 +1154965t^2 -1995487t +1486578$ & 5106180\\
        \hline
        9 & $t^9 -54t^8 +1350t^7 -20286t^6 +200025t^5 -1331022t^4 +5932143t^3 -16952157t^2 +27979203t -20158695$ & 72574936\\
        \hline
        10 & $t^{10} -65t^9 +1980t^8 -36840t^7 +459585t^6 -3986031t^5 +24172575t^4 -100548090t^3 +272771475t^2 -432836011t +302751327$ & 1137563980\\
        \hline
    \end{tabularx}
    \caption{Characteristic polynomial and the number of regions of $\mathcal{BT}_n$ for $n \leq 10$.}
    \label{tab:my_label}
\end{table}

\section{The signed ordered partitions}\label{section3}
In this section we prove a bijection between regions of $\mathcal{BT}_n$, and certain ordered partitions of $[-n,n]\setminus \{0\} \cup \{-\frac{1}{2},\frac{1}{2}\}$ (the notation $[-n, n]$ is used for the set $\{-n, -n+1,\dots, n-1, n \}$).
We will denote $-x_i$ by $x_{-i}$ for all $i \in [n]$.
Let $R$ be a region of $\mathcal{BT}_n$.
We write
\begin{enumerate}
    \item $i \prec_R -j$ if $x_i+x_j < 0$ in $R$ where $i \neq j$ in $[n]$.
    \item $i \succ_R -j$ if $x_i+x_j > 0$ in $R$ where $i \neq j$ in $[n]$.
    \item $i \succ_R \frac{1}{2}$ if $x_i > \frac{1}{2}$ in $R$ where $i \in [-n,n] \setminus \{0\}$.
    \item Similarly define $i \prec_R \frac{1}{2}$, $i \succ_R -\frac{1}{2}$ and $i \prec_R -\frac{1}{2}$ for any $i \in [-n,n] \setminus \{0\}$.
\end{enumerate}
So $\prec_R$ is a relation on $[-n,n]\setminus \{0\} \cup \{-\frac{1}{2},\frac{1}{2}\}$ where comparable elements are
\begin{enumerate}
    \item Elements of $[-n,n] \setminus \{0\}$ of different signs and different absolute values.
    \item Elements of $[-n,n]\setminus \{0\}$ and $\pm \frac{1}{2}$.
\end{enumerate}

\begin{remark}
The reader can consider the relation $i \prec_R -j$ as $\overset{+}{i}$ appearing before $\overset{-}{j}$ in a signed permutation on $[n]$ and similarly for other such relations.
The equivalence defined in the following lemma corresponds to choosing a signed permutation representative for each region.
This is similar to the way Spiro \cite{spiro20} associated `threshold pairs in standard form' to labeled threshold graphs.
\end{remark}

\begin{lemma}\label{equi}
Let $\sim$ be the relation defined on the set $N=[-n,n]\setminus \{0\}$ as

$a \sim b$ if the following hold:
\begin{enumerate}
    \item $a$ and $b$ are of the same sign, and
    \item there does not exist any $c \in N \cup \{-\frac{1}{2},\frac{1}{2}\}$ such that $a \prec_R c \prec_R b$ or $b \prec_R c \prec_R a$.
\end{enumerate}
Then, $\sim$ is an equivalence relation on $N$.
\end{lemma}
\begin{proof}
It is clear that $a \sim a$ for any $a \in N$ and that $a \sim b$ implies that $b \sim a$ for any $a,b \in N$.
Let $a \sim b$ and $b \sim c$ for some distinct $a,b,c \in N$.
So $a,b,c$ are of the same sign.
By definition of $\prec_R$ and $\sim$, the only possible $d \in N \cup \{-\frac{1}{2},\frac{1}{2}\}$ such that $a \prec_R d \prec_R c$ is $-b$.
We must show that this is not possible (a similar argument works if $c \prec_R -b \prec_R a$).
Suppose $a \prec_R -b \prec_R c$.
We then have $-c \prec_R b \prec_R -a$.
If $a \prec_R -c$, we will have $a\prec_R -c \prec_R b$, which contradicts $a \sim b$.
So we must have $-c \prec_R a$.
But this gives $-a\prec_R c$ and hence $b \prec_R -a \prec_R c$, which is a contradiction to $b \sim c$.
Hence $\sim$ is an equivalence on $N$.
\end{proof}

\begin{definition}
The equivalence classes of $\sim$ defined in Lemma \ref{equi} will be called \textit{boxed threshold blocks} (corresponding to the region $R$).
Positive blocks refer to those blocks that contain positive numbers and similarly we define negative blocks.
\end{definition}

\begin{remark}
Since all the elements of a boxed threshold block are of the same sign, we sometimes consider a block as a subset of $[n]$ with a sign ($+$ or $-$) assigned to it.
\end{remark}

The proof of the following lemma follows from the definition of the equivalence.
\begin{lemma}
If $B$ is a boxed threshold block then so is $-B = \{ -b : b \in B\}$.
\end{lemma}

\begin{proposition}\label{totord}
Define an order $<$ on the set of boxed threshold blocks along with $\pm \frac{1}{2}$ as follows:
\begin{enumerate}
    \item $B < B'$ where $B,B'$ are boxed threshold blocks if there exists a sequence $c_0,c_1,\dots,c_k$ of elements in $N \cup \{-\frac{1}{2},\frac{1}{2}\}$ such that $c_0 \prec_R c_1 \prec_R \cdots \prec_R c_k$, $c_0 \in B$ and $c_k \in B'$.
    \item $-\frac{1}{2} < \frac{1}{2}$.
    \item $B < \frac{1}{2}$ if $b \prec_R \frac{1}{2}$ for some $b \in B$. Similarly define other relations between blocks and $\pm \frac{1}{2}$.
\end{enumerate}
This is a total order in all cases except when there is a unique $i \in [n]$ such that $-\frac{1}{2} \prec_R i \prec_R \frac{1}{2}$.
\end{proposition}
\begin{proof}
The transitivity of this order is straightforward.
If we show that $B<B$ is not possible for any block, the order is well-defined.
Suppose $c_0 \prec_R c_1 \prec_R \cdots \prec_R c_k$ where $c_0,c_k \in B$.
If $c_1 \neq -c_k$ we get a contradiction to $c_0 \sim c_k$ and similarly if $c_{k-1} \neq -c_1$.
So we must have $c_1 = -c_k$ and $c_{k-1}=-c_1$.
But then we get $c_1 \prec_R -c_k$ and $-c_1 \prec_R c_k$, which is a contradiction.

We will now show that this order is a total order in all cases except when there is unique $i \in [n]$ such that $-\frac{1}{2} \prec_R i \prec_R \frac{1}{2}$.
The relationship of any block with $\pm \frac{1}{2}$ is always defined.
So we only have to check that any two blocks are comparable.
Let $B,B'$ be boxed threshold blocks.
If $B,B'$ are distinct blocks of the same sign, they are comparable by definition of the equivalence relation.
If they are of opposite sign and not of the form $\{i\},\{-i\}$ for some $i \in [n]$, they are comparable since there is some $a \in B$ and $b \in B'$ of opposite signs such that $|a| \neq |b|$.
So we have to deal with the case $B=\{i\}$ and $B'=\{-i\}$ for some $i \in [n]$.
If $i \succ_R \frac{1}{2}$ or $i \prec_R -\frac{1}{2}$, $\{-i\}$ and $\{i\}$ are comparable.

Suppose $-\frac{1}{2} \prec_R i \prec_R \frac{1}{2}$ and $\{i\}$ and $\{-i\}$ are blocks and they are not comparable.
We have already seen that all the positive blocks are comparable and so are all the negative blocks.
Let the order on the positive blocks be $B_1 < B_2 < \cdots < B_k$ and hence the order on the negative blocks is $-B_k < \cdots < -B_2 < -B_1$.

Suppose $B_1 = \{i\}$.
Since $i$ is not the only number satisfying $-\frac{1}{2} \prec_R i \prec_R \frac{1}{2}$, we have $b \prec_R \frac{1}{2}$ for all $b \in B_2$.
Taking some $b \in B_2$, since $i \notin B_2$, there is some $k \in [n]$ such that $i \prec_R -k \prec_R b$.
Since $\{-i\}$ is a block, and $k \neq i$, there is some $l \in [n]$ such that $-i \prec_R l \prec_R -k$ (we cannot have $-k \prec_R l \prec_R -i$ since this would mean $\{i\}$ and $\{-i\}$ are comparable).
But this gives $k \prec_R -l \prec_R i$, which contradicts the fact that there is no positive block less than $\{i\}$.
We get a similar contradiction if $B_k=\{i\}$.

Suppose $B_m=\{i\}$ for some $m \in [2,k-1]$.
Take some $b_p \in B_{m-1}$ and $b_s \in B_{m+1}$.
There are three possibilities:
\begin{enumerate}
    \item There are $k_p,k_s \in [n]$ such that
\begin{equation*}
    b_p \prec_R -k_p \prec_R i \prec_R -k_s \prec_R b_s.
\end{equation*}Since $k_p,k_s \neq i$, and $\{-i\}$ is a block, there are $c_p,c_s \in [n]$ such that
\begin{align*}
    -k_p &\prec_R c_p \prec_R -i\\
    -i &\prec_R c_s \prec_R -k_s
\end{align*}where the other possibilities are not possible since $\{i\}$ and $\{-i\}$ are not comparable. So we get
\begin{equation*}
    b_p \prec_R -k_p \prec_R c_p \prec_R -i \prec_R c_s \prec_R -k_s \prec_R b_s.
\end{equation*}But this means that $c_p$ and $c_s$ are in blocks between $B_{m-1}$ and $B_{m+1}$ and hence in $\{i\}$, which is a contradiction.
\item There is no $k_s \in [n]$ such that $i \prec_R -k_s <b_s$. So we must have $i \prec_R \frac{1}{2} \prec_R b_s$, $-\frac{1}{2} \prec_R b_p \prec_R \frac{1}{2}$ and $k_p \in [n]$ such that
\begin{equation*}
    b_p \prec_R -k_p \prec_R i \prec_R \frac{1}{2} \prec_R b_s.
\end{equation*}This is because $i$ satisfies $-\frac{1}{2} \prec_R i \prec_R \frac{1}{2}$ and is not the only such number. Since $k_p \neq i$, and $\{-i\}$ is a block, there is some $c_p \in [n]$ such that
\begin{equation*}
    -k_p \prec_R c_p \prec_R -i
\end{equation*}where the other possibility is not possible since $\{i\}$ and $\{-i\}$ are not comparable. So we get
\begin{equation*}
    -\frac{1}{2} \prec_R b_p \prec_R -k_p \prec_R c_p \prec_R -i
\end{equation*}which, along with previous observations, gives
\begin{equation*}
    i \prec_R -c_p \prec_R k_p \prec_R \frac{1}{2} \prec_R b_s.
\end{equation*}But this means $k_p$ is in a positive block after $\{i\}$ but before $B_{m+1}$, which is a contradiction.
\item The case when there is no $k_p \in [n]$ such that $b_p \prec_R -k_p \prec_R i$ is handled similarly.
\end{enumerate}

In the case when there is unique $i \in [n]$ such that $-\frac{1}{2} \prec_R i \prec_R \frac{1}{2}$, the only order that is not defined is between the blocks $\{i\}$ and $\{-i\}$ and the order is of the form
\begin{equation}\label{onlyonemiddle}
    -B_k < \cdots < -B_2 < -\frac{1}{2} < \{i\},\{-i\} < \frac{1}{2} < B_2 < \cdots < B_k
\end{equation}where $B_2,\dots,B_k$ are blocks (not necessarily positive).
This can be proved using similar arguments as before.
\end{proof}

If there is a  unique $i \in [n]$ such that $-\frac{1}{2} \prec_R i \prec_R \frac{1}{2}$ and the blocks are ordered as in \eqref{onlyonemiddle}, we make the convention that $\{-i\} < \{i\}$ if $B_2$ is a positive block and $\{i\} < \{-i\}$ if $B_2$ is a negative block.
Once this convention is made, we get a total order on the boxed threshold blocks, $\frac{1}{2}$ and $-\frac{1}{2}$ in all cases.
The proof of the following lemma follows from the definition of the order.
\begin{lemma}\label{negord}
For any blocks $B, B'$, $B < B'$ implies that $-B' < -B$.
Similarly, other usual properties of taking the negative hold; such as $B < \frac{1}{2}$ implies that $-\frac{1}{2} < -B$.
\end{lemma}

\begin{proposition}
The total order associated to a region (defined in Proposition \ref{totord}) is always of one of the following forms:
\begin{enumerate}
    \item $-B_k < \cdots < -B_2 < -B_1 < -\frac{1}{2} < \frac{1}{2} < B_1 < B_2 < \cdots < B_k$ where $B_i$ and $B_{i+1}$ are of opposite signs for all $i \in [k-1]$.
    \item $-B_k < \cdots < -B_{l+1} < -\frac{1}{2}< -B_l < \cdots < -B_1 < B_1 < \cdots < B_l < \frac{1}{2} < B_{l+1} < \cdots < B_k$ where the size of $B_1$ is greater than 1 and $B_i$ and $B_{i+1}$ are of opposite signs for all $i \in [l-1]$ and $i \in [l+1,k-1]$.
    \item $-B_k < \cdots < -B_2 < -\frac{1}{2} < -B_1 < B_1 < \frac{1}{2} < B_2 < \cdots < B_k$ where the size of $B_1$ is 1, $B_1$ and $B_2$ are of the same sign, and $B_i$ and $B_{i+1}$ are of opposite signs for all $i \in [2,k-1]$.
\end{enumerate}
The association of such an order is a bijection between the regions of $\mathcal{BT}_n$ and total orders of the types mentioned above.
\end{proposition}
\begin{proof}
The fact that the first half of the order is the negative of the second follows from Lemma \ref{negord}.
By the definition of blocks, for any two blocks of the same sign which are in the same position with respect to $\pm \frac{1}{2}$, there is some block of opposite sign between them.

The first form is when there is no $i \in [n]$ such that $-\frac{1}{2} \prec_R i \prec_R \frac{1}{2}$.
The third form is by the convention we made at the end of the proof of Proposition \ref{totord}, when there is a unique $i \in [n]$ such that $-\frac{1}{2} \prec_R i \prec_R \frac{1}{2}$.

So we have to show that when there is more than one $i \in [n]$ such that $-\frac{1}{2} \prec_R i \prec_R \frac{1}{2}$, the block $B_1$ has size greater than 1.
Suppose $B_1=\{a\}$ for some $a \in N$.
Let $b \in B_2$.
$b$ has the same sign as $-a$ and is in a different block.
Hence there is some $k \in N$ of opposite sign to $-a$ and $b$ such that $-a \prec_R k \prec_R b$.
But this means $k$ is in a block between $\{-a\}$ and $B_2$ and hence in $\{a\}$, which is a contradiction since $k \neq a$.

We will now show that the association of such an order is a bijection between the regions of $\mathcal{BT}_n$ and total orders of the types mentioned above.
First note that, from the total order associated to a region, we can get back the inequalities describing the region as follows:
For any $i \neq j$ in $[n]$, $x_i+x_j > 0$ if and only if the block containing $-j$ is before that containing $i$ and the relationship between $x_i$ and $\pm \frac{1}{2}$ is obtained in the same way.
On the other hand, given an order of one of the forms given above, choosing some real numbers $0< c_1 < c_2 < \cdots < c_k$ such that $c_i < \frac{1}{2}$ if $B_i < \frac{1}{2}$ and $c_i > \frac{1}{2}$ if $B_i > \frac{1}{2}$ and putting $x_a=c_i$ for all $a \in B_i$ for all $i \in [k]$, gives a point satisfying the required inequalities.
It can also be shown that different such orders correspond to different regions.
\end{proof}

Hence, to count the number of regions of $\mathcal{BT}_n$, we just have to count the number of orders of the forms mentioned in Proposition \ref{totord}.
Note that the first half of the order is the negative of the second, so we just consider the second half.
That is, we count orders of the following types, where in all cases, $B_1,\dots,B_k$ is a partition of $[n]$ with a sign assigned to each block:
\begin{enumerate}
    \item $\frac{1}{2} < B_1 < B_2 < \cdots < B_k$ where $B_i$ and $B_{i+1}$ are of opposite signs for all $i \in [k-1]$.
    \item $B_1 < \cdots < B_l < \frac{1}{2} < B_{l+1} < \cdots < B_k$ where the size of $B_1$ is greater than 1 and $B_i$ and $B_{i+1}$ are of opposite signs for all $i \in [l-1]$ and $i \in [l+1,k-1]$.
    \item $B_1 < \frac{1}{2} < B_2 < \cdots < B_k$ where the size of $B_1$ is 1, $B_1$ and $B_2$ are of the same sign, and $B_i$ and $B_{i+1}$ are of opposite signs for all $i \in [2,k-1]$.
\end{enumerate}

\begin{proposition}\label{count}
The total number of orders of the forms mentioned above is
\begin{equation*}
    4 \cdot a(n)+\sum_{k=1}^{n} 4  (k!-(k-1)!) (k \cdot S(n,k)-n \cdot S(n-1,k-1)).
\end{equation*}Here $a(n)$ is the $n^{th}$ ordered Bell number.
\end{proposition}
\begin{proof}
We will count the number of orders of each of the above forms.
\begin{enumerate}
    \item In the first case, we just have to define an ordered partition of $[n]$ and assign alternating signs to them.
    The number of ways this can be done is
    \begin{equation*}
        \sum_{k=1}^{n}\sum\limits_{\substack{(a_1,\dots,a_k)\\a_1+\cdots+a_k=n}}2\cdot \dfrac{n!}{a_1!\cdots a_k!} = 2 \cdot a(n).
    \end{equation*}
    \item In the second case, we consider two sub-cases:
    \begin{enumerate}
        \item There is no block after $\frac{1}{2}$.
        In this case, we have to define an ordered partition of $[n]$ whose first part has size greater that $1$ and assign alternating signs to them.
        The number of ways this can be done is
        \begin{equation*}
            \sum_{k=1}^{n-1}\sum\limits_{\substack{(a_1,\dots,a_k)\\a_1+\cdots+a_k=n,\ a_1 \neq 1}}2\cdot \dfrac{n!}{a_1!\cdots a_k!} = 2 (a(n) - n \cdot a(n-1))
        \end{equation*}
        where the equality is because the number of ordered partitions of $[n]$ with first block having size $1$ is $n \cdot a(n-1)$.
        \item There is some block after $\frac{1}{2}$.
        In this case, we have to again define an ordered partition of $[n]$ whose first part has size greater that $1$.
        But we then have to choose a spot between two blocks to place $\frac{1}{2}$ and then choose a sign for the first block and the block after $\frac{1}{2}$.
        The number of ways this can be done is
        \begin{equation*}
            \sum_{k=1}^{n-1}\sum\limits_{\substack{(a_1,\dots,a_k)\\a_1+\cdots+a_k=n,\ a_1 \neq 1}}4 (k-1) \dfrac{n!}{a_1!\cdots a_k!}.
        \end{equation*}Making the following substitution for all $k \in [n-1]$
        \begin{align*}
            \sum\limits_{\substack{(a_1,\dots,a_k)\\a_1+\cdots+a_k=n,\ a_1 \neq 1}}\dfrac{n!}{a_1!\cdots a_k!} &= \sum\limits_{\substack{(a_1,\dots,a_k)\\a_1+\cdots+a_k=n}}\dfrac{n!}{a_1!\cdots a_k!}\hspace{0.2cm} - \sum\limits_{\substack{(1,a_2,\dots,a_k)\\1+a_2+\cdots+a_k=n}}\dfrac{n!}{1!a_2!\cdots a_k!}\\[0.1cm]
            &= k! \cdot S(n,k)-n \cdot (k-1)! \cdot S(n-1,k-1)
        \end{align*}we get that the initial expression is the same as
        \begin{equation*}
            \sum_{k=1}^{n} 4 (k!-(k-1)!)(k \cdot S(n,k)-n \cdot S(n-1,k-1)).
        \end{equation*}
    \end{enumerate}
    \item In the third case, we have to choose the element of $[n]$ in $B_1$ and then define an ordered partition of the remaining $(n-1)$ elements and assign alternating signs to them.
    Since we want $B_1$ and $B_2$ to have the same sign, we just need to assign a sign to $B_2$.
    So, the number of orders of the third type is
    \begin{equation*}
        n \cdot \sum_{k=1}^{n-1} \sum\limits_{\substack{(a_1,\dots,a_k)\\a_1+\cdots+a_k=n-1}}2\cdot \dfrac{(n-1)!}{a_1!\cdots a_k!} = n \cdot 2 \cdot a(n-1).
    \end{equation*}
\end{enumerate}Adding up the counts made for each form gives us the required result.
\end{proof}

So, from the observations made above, we have proved the following theorem:
\begin{theorem}
The number of regions of $\mathcal{BT}_n$ is
\begin{align*}
    4 \cdot a(n)+\sum_{k=1}^{n} 4 (k!-(k-1)!) (k \cdot S(n,k)-n \cdot S(n-1,k-1)).
\end{align*}
\end{theorem}
Similar arguments can be applied to the threshold arrangement $\mathcal{T}_n$.
\begin{proposition}\label{proposition:thresholdregions}
The regions of $\mathcal{T}_n$ are in bijection with ordered partitions of $[-n,n] \setminus \{0\}$ of the form
\begin{equation*}
    -B_k < \cdots < -B_2 < -B_1 < B_1 < B_2 < \cdots < B_k
\end{equation*}where all elements of each block have the same sign, the size of $B_1$ is greater than $1$ and $B_i$ and $B_{i+1}$ are of opposite signs for all $i \in [k-1]$.
\end{proposition}

The bijection in Proposition \ref{proposition:thresholdregions} is defined just as was done for $\mathcal{BT}_n$.
That is, the region associated to such an order satisfies $x_i+x_j > 0$ if and only if $-j$ appears before $i$ in the order.
Again, such orders are completely specified by their second half, which are ordered partition of $[n]$ with first block size greater than $1$ and a sign assigned to the first block (and alternate signs for consecutive blocks).
So, we get the following theorem:
\begin{theorem}\label{theorem:counting of threshold regions}
The number of regions of $\mathcal{T}_n$ is
\begin{equation*}
    2(a(n)-n \cdot a(n-1)).
\end{equation*}
\end{theorem}

\begin{remark}
It can be shown that the orders considered in \cref{proposition:thresholdregions} are in bijection with the set of threshold pairs (of size $n$) in standard form\footnote{A pair $(\pi, \omega)$ is a \emph{threshold pair} (of size $n$) if $\pi$ is a permutation of size $n$ and $\omega$ is a word in $\{+1,-1\}^n$. A threshold pair $(\pi, \omega)$ of size $n \geq 2$ is in \emph{standard form} if $\omega_1 = \omega_2$ and if $\omega_i = \omega_{i+1}$ implies $\pi_i < \pi_{i+1}$ for all $1 \leq i < n$.} considered by Spiro \cite{spiro20}.
In fact, he showed that the threshold pairs are in bijection with the labeled threshold graphs. 
\end{remark}

The known formula for the number of labeled threshold graphs is in terms of Eulerian numbers. 
Hence for the sake of completeness, we now show that the formula in \Cref{theorem:counting of threshold regions} is the same as the one containing Eulerian numbers. 
We first recall a few identities related to  Eulerian numbers $A(n,k)$ and the ordered Bell number $a(n)$.
A quick reference for these identities is B\'ona's book \cite[Section 1.1]{bonabook}.

\begin{itemize}
\item $a(n)=\sum\limits_{k=0}^{n-1} 2^k \cdot A(n,k)$
\item $A(n,0)=1$, $A(n,n-1)=1$ and $A(n,k)=0$ for all $k \geq n$.
\item $A(n,k)=(n-k) A(n-1,k-1)+(k+1) A(n-1,k)$ for $k\geq 1$.
\end{itemize}

\begin{proposition}
We have the following equality
\[r(\mathcal{T}_n) = \sum\limits_{k=1}^{n-1} 2^k (n-k) A(n-1,k-1).\]
\end{proposition}
\begin{proof} 
\noindent From \cref{theorem:counting of threshold regions}, we have  
\begin{equation*}
\begin{split}
r(\mathcal{T}_n) &= 2(a(n)-n \cdot a(n-1))\\
& = 2 \cdot \Bigg(\sum\limits_{k=0}^{n-1} 2^k \cdot A(n,k)- n \cdot \Big( \sum\limits_{k=0}^{n-2} 2^k \cdot A(n-1,k)\Big)\Bigg)\\
 \frac{r(\mathcal{T}_n)}{2} & = A(n,0)+\sum\limits_{k=1}^{n-1} 2^k \big((n-k) A(n-1,k-1)+(k+1) A(n-1,k)\big)\\
& \hspace*{0.5cm}- n \Big( \sum\limits_{k=0}^{n-2} 2^k \cdot A(n-1,k)\Big) \\
& = \sum\limits_{k=1}^{n-1} 2^k (n-k) A(n-1,k-1) + \Big(A(n,0)+\sum\limits_{k=1}^{n-1} 2^k (k+1) A(n-1,k)\Big) \\
& \hspace*{0.5cm}- \sum\limits_{k=0}^{n-2} n \cdot 2^k \cdot A(n-1,k)
\end{split}
\end{equation*}
Since $A(n-1,n-1)=0$,
\begin{equation*}
\begin{split}
\frac{r(\mathcal{T}_n)}{2} & = \sum\limits_{k=1}^{n-1} 2^k (n-k) A(n-1,k-1) + \sum\limits_{k=0}^{n-2} 2^k (k+1) A(n-1,k)- \sum\limits_{k=0}^{n-2} n \cdot 2^k \cdot A(n-1,k)\\
&= \sum\limits_{k=1}^{n-1} 2^k (n-k) A(n-1,k-1) - \sum\limits_{k=0}^{n-2} (n-k-1) 2^k \cdot A(n-1,k)
\end{split}
\end{equation*}
Replacing $k+1$ by $t$ in the second sum, we get
\begin{equation*}
\begin{split}
r(\mathcal{T}_n) & = 2 \cdot \Bigg(\sum\limits_{k=1}^{n-1} 2^k (n-k) A(n-1,k-1) - \sum\limits_{t=1}^{n-1} 2^{t-1} (n-t) A(n-1,t-1)\Bigg)\\
& = \sum\limits_{k=1}^{n-1} 2^k (n-k) A(n-1,k-1).
\end{split}
\end{equation*}
\end{proof}

\section{The colored threshold graphs}\label{section4}

Before defining the colored threshold graphs that are in bijection with the regions of the boxed threshold arrangement, we recall the bijection between regions of the threshold arrangement and labeled threshold graphs.

\begin{definition}
A \textit{threshold graph} is defined recursively as follows:
\begin{enumerate}
    \item The empty graph is a threshold graph.
    \item A graph obtained by adding an isolated vertex to a threshold graph is a threshold graph.
    \item A graph obtained by adding a vertex adjacent to all vertices of a threshold graph is a threshold graph.
\end{enumerate}
\end{definition}

\begin{definition}
A \textit{labeled threshold graph} is a threshold graph having $n$ vertices with vertices labeled distinctly using $[n]$.
\end{definition}

Such labeled threshold graphs can be specified by a signed permutation of $[n]$, that is, a permutation of $[n]$ with a sign associated to each number.
The signed permutation $i_1\ i_2\ \cdots\ i_n$ would correspond to the labeled threshold graph obtained by adding vertices labeled $|i_1|,|i_2|,\dots,|i_n|$ in order where a positive $i_k$ means that $|i_k|$ is added adjacent to all previous vertices and a negative $i_k$ means that it is added isolated to the previous vertices.
A maximal string of positive numbers or negative numbers in a signed permutation will be called a block.

\begin{example}
The labeled threshold graph associated to the signed permutation on $[5]$ given by $\overset{-}{2}\overset{-}{3}\overset{+}{1}\overset{+}{4}\overset{-}{5}$ is shown in Figure \ref{tgconst}.
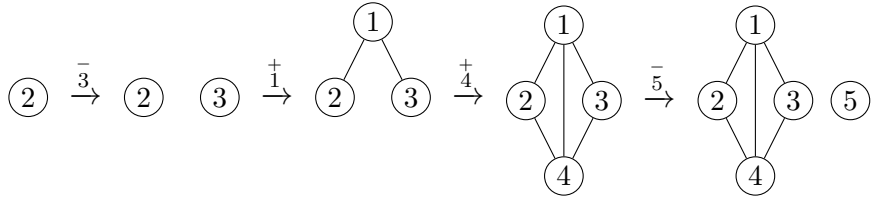
\begin{figure}[H]
    \centering
    \begin{tikzpicture}[scale=0.5]
    \node (2) [circle,draw=black,inner sep=2pt] at (0,0) {\small $2$};
    \node (4) at (0.75,-2) {\small \color{white} 4};
    \node at (1.5,0.5) {$\xrightarrow{\overset{-}{3}}$};
    \end{tikzpicture}
    \begin{tikzpicture}[scale=0.5]
        \node (2) [circle,draw=black,inner sep=2pt] at (0,0) {\small $2$};
        \node (3) [circle,draw=black,inner sep=2pt] at (2,0) {\small $3$};
        \node (4) at (0.75,-2) {\small \color{white} 4};
        \node at (3.5,0.5) {$\xrightarrow{\overset{+}{1}}$};
    \end{tikzpicture}
    \begin{tikzpicture}[scale=0.5]
        \node (2) [circle,draw=black,inner sep=2pt] at (0,0) {\small $2$};
        \node (3) [circle,draw=black,inner sep=2pt] at (2,0) {\small $3$};
        \node (1) [circle,draw=black,inner sep=2pt] at (1,2) {\small $1$};
        \node (4) at (0.75,-2) {\small \color{white} 4};
        \node at (3.5,0.5) {$\xrightarrow{\overset{+}{4}}$};
        \draw (3)--(1)--(2);
    \end{tikzpicture}
    \begin{tikzpicture}[scale=0.5]
        \node (2) [circle,draw=black,inner sep=2pt] at (0,0) {\small $2$};
        \node (3) [circle,draw=black,inner sep=2pt] at (2,0) {\small $3$};
        \node (1) [circle,draw=black,inner sep=2pt] at (1,2) {\small $1$};
        \draw (3)--(1)--(2);
        \node at (3.5,0.5) {$\xrightarrow{\overset{-}{5}}$};
        \node (4) [circle,draw=black,inner sep=2pt] at (1,-2) {\small $4$};
        \draw (1)--(4);
        \draw (2)--(4);
        \draw (3)--(4);
    \end{tikzpicture}
    \begin{tikzpicture}[scale=0.5]
        \node (2) [circle,draw=black,inner sep=2pt] at (0,0) {\small $2$};
        \node (3) [circle,draw=black,inner sep=2pt] at (2,0) {\small $3$};
        \node (1) [circle,draw=black,inner sep=2pt] at (1,2) {\small $1$};
        \draw (3)--(1)--(2);
        \node (4) [circle,draw=black,inner sep=2pt] at (1,-2) {\small $4$};
        \node (5) [circle,draw=black,inner sep=2pt] at (3.5,0) {\small $5$};
        \draw (1)--(4);
        \draw (2)--(4);
        \draw (3)--(4);
    \end{tikzpicture}
    \caption{Construction of threshold graph corresponding to $\overset{-}{2}\overset{-}{3}\overset{+}{1}\overset{+}{4}\overset{-}{5}$.}
    \label{tgconst}
\end{figure}
\end{example}

The following facts can be verified:
\begin{enumerate}
    \item The sign of the first number in the permutation does not matter and hence we can make the first block have size greater than 1.
    \item Elements in the same block can be reordered.
\end{enumerate}
Hence, labeled threshold graphs can be specified by an ordered partition of $[n]$ with first block size greater than 1 and alternating signs assigned to the blocks.
In fact, this association is a bijection.

Given a threshold graph $G_1$, we can obtain this alternating signed ordered partition of $[n]$ as follows:
Since $G_1$ is a threshold graph, it has at least one isolated vertex or at least one vertex that is adjacent to all other vertices.
If it has an isolated vertex, set $D_1$ to be the set of all isolated vertices, assign it a negative sign and set $G_2$ to be the graph obtained by deleting all the vertices of $D_1$ from $G_1$.
If $G_1$ has at least one vertex adjacent to all other vertices, set $D_1$ to be the set of all such vertices, assign it a positive sign and set $G_2$ to be the graph obtained by deleting all the vertices of $D_1$ from $G_1$.
We repeat this process until we obtain a graph $G_k$ which is complete, in which case we set $D_k$ to be all vertices of $G_k$ and assign it a positive sign, or $G_k$ has no edges, in which case we set $D_k$ to be all vertices of $G_k$ and assign it a negative sign.
Then set $B_i=D_{k-i+1}$ and assign it the same sign as $D_{k-i+1}$ for all $i \in [k]$.
The signed ordered partition $B_1 , \dots, B_k$ is the one associated to $G_1$.

\begin{example} Figure \ref{tgdeconst} shows an example of obtaining the signed blocks from a threshold graph. The corresponding signed ordered partition for this example is $\overset{-}{\{2,3\}}\overset{+}{\{1,4\}}\overset{-}{\{5\}}$.
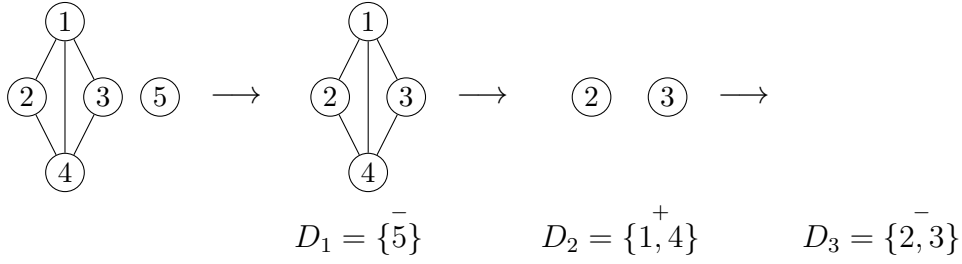
\begin{figure}[H]
    \centering
    \begin{tikzpicture}[scale=0.5]
        \node (2) [circle,draw=black,inner sep=2pt] at (0,0) {\small $2$};
        \node (3) [circle,draw=black,inner sep=2pt] at (2,0) {\small $3$};
        \node (1) [circle,draw=black,inner sep=2pt] at (1,2) {\small $1$};
        \draw (3)--(1)--(2);
        \node (4) [circle,draw=black,inner sep=2pt] at (1,-2) {\small $4$};
        \node (5) [circle,draw=black,inner sep=2pt] at (3.5,0) {\small $5$};
        \node at (5.5,0) {$\longrightarrow$};
        \draw (1)--(4);
        \draw (2)--(4);
        \draw (3)--(4);
        \node at (0.75,-3.5) {\color{white}$D_2=\overset{+}{\{1,4\}}$};
    \end{tikzpicture}
    \begin{tikzpicture}[scale=0.5]
        \node (2) [circle,draw=black,inner sep=2pt] at (0,0) {\small $2$};
        \node (3) [circle,draw=black,inner sep=2pt] at (2,0) {\small $3$};
        \node (1) [circle,draw=black,inner sep=2pt] at (1,2) {\small $1$};
        \draw (3)--(1)--(2);
        \node at (4,0) {$\longrightarrow$};
        \node (4) [circle,draw=black,inner sep=2pt] at (1,-2) {\small $4$};
        \draw (1)--(4);
        \draw (2)--(4);
        \draw (3)--(4);
        \node at (0.75,-3.5) {$D_1=\overset{-}{\{5\}}$};
    \end{tikzpicture}
    \begin{tikzpicture}[scale=0.5]
        \node (2) [circle,draw=black,inner sep=2pt] at (0,0) {\small $2$};
        \node (3) [circle,draw=black,inner sep=2pt] at (2,0) {\small $3$};
        \node (4) at (0.75,-2) {\small \color{white} 4};
        \node at (4,0) {$\longrightarrow$};
        \node at (0.75,-3.5) {$D_2=\overset{+}{\{1,4\}}$};
    \end{tikzpicture}
    \begin{tikzpicture}[scale=0.5]
        \node at (0.75,-3.5) {$D_3=\overset{-}{\{2,3\}}$};
    \end{tikzpicture}
    \caption{Obtaining blocks from a threshold graph.}
    \label{tgdeconst}
\end{figure}
\end{example}

Hence, regions of $\mathcal{T}_n$ and labeled threshold graphs on $n$ vertices are both in bijection with ordered partitions of $[n]$ with first block size greater than 1 and alternating signs assigned to the blocks.
Hence we obtain a bijection between regions of $\mathcal{T}_n$ and labeled threshold graphs on $n$ vertices.
By combining the definitions of the two bijections we see that to a labeled threshold graph on $n$ vertices we assign the region where $x_i+x_j > 0$ if and only if there is an edge between $i$ and $j$.

This can be proved as follows: If $-B_k < \cdots < -B_1 < B_1 < \cdots <B_k$ is the threshold block order corresponding to some region $R$ of $\mathcal{T}_n$, $-j \prec_R i$ for some $i \neq j$ in $[n]$ if and only if one of the following holds:
\begin{enumerate}
    \item $-j$ and $i$ both appear in $B_1, \dots, B_k$ with $-j$ appearing first.
    \item $-j$ appears in $-B_k,\dots,-B_1$ and $i$ appears in $B_1, \dots, B_k$.
    \item $-j$ and $i$ both appear in $-B_k, \dots, -B_1$ with $-j$ appearing first.
\end{enumerate}
Equivalently, one of the following holds:
\begin{enumerate}
    \item $-j$ and $i$ both appear in $B_1, \dots, B_k$ with $-j$ appearing first.
    \item $i$ and $j$ both appear in $B_1,\dots,B_k$.
    \item $-i$ and $j$ both appear in $B_1,\dots,B_k$ with $-i$ appearing first.
\end{enumerate}
But this is precisely the condition for there to be an edge between $i$ and $j$ in the threshold graph corresponding to $B_1<\cdots<B_k$.

We now move on to the boxed threshold arrangement.
\begin{definition}
A \textit{colored threshold graph} is defined recursively as follows:
\begin{enumerate}
    \item The empty graph is a colored threshold graph.
    \item A graph obtained by adding an isolated vertex to a colored threshold graph is a colored threshold graph. If there are colored vertices in the initial colored threshold graph, the new vertex should be colored red. If not, the new vertex can be left uncolored or colored red.
    \item A graph obtained by adding a vertex adjacent to all vertices of a colored threshold graph is a colored threshold graph. 
    If there are colored vertices in the initial colored threshold graph, the new vertex should be colored blue. 
    If not, the new vertex can be left uncolored or colored blue.
\end{enumerate}
\end{definition}

\begin{definition}
A \textit{labeled colored threshold graph} is a colored threshold graph with $n$ vertices with the vertices labeled distinctly with elements of $[n]$.
\end{definition}

Just as for threshold graphs, labeled colored threshold graphs can be represented as a signed permutation.
However, we also have to specify if and when the coloring of the vertices starts.
This is done by using the symbol $\frac{1}{2}$.
Having $\frac{1}{2}$ at the end of the signed permutation means that none of the vertices should be colored.

\begin{example}
The sequence $\overset{+}{2} \frac{1}{2}\overset{+}{1}\overset{+}{3}\overset{-}{4}\overset{-}{5}$ corresponds to the graph shown in Figure \ref{ctgconst}.
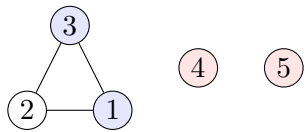
\begin{figure}[H]
    \centering
    \begin{tikzpicture}[scale=0.75]
        \node (2) [circle,draw=black,inner sep=2pt] at (0,0) {\small $2$};
        \node (3) [circle,draw=black,fill=blue!10,inner sep=2pt] at (1.5,0) {\small $1$};
        \node (1) [circle,draw=black,fill=blue!10,inner sep=2pt] at (0.75,1.5) {\small $3$};
        \draw (3)--(1)--(2)--(3);
        \node (4) [circle,draw=black,fill=red!10,inner sep=2pt] at (3,0.75) {\small $4$};
        \node (5) [circle,draw=black,fill=red!10,inner sep=2pt] at (4.5,0.75) {\small $5$};
    \end{tikzpicture}
    \caption{Labeled colored threshold graph corresponding to $\overset{+}{2} \frac{1}{2}\overset{+}{1}\overset{+}{3}\overset{-}{4}\overset{-}{5}$.}
    \label{ctgconst}
\end{figure}
\end{example}

Using similar observations about these sequences associated to labeled colored threshold graphs as done for labeled threshold graphs, we get that labeled colored threshold graphs are in bijection with orders of the forms counted in \Cref{count}.
Since these orders also correspond to region of $\mathcal{BT}_n$, we get a bijection between labeled colored threshold graphs with $n$ vertices and regions of $\mathcal{BT}_n$.
Just as before, the inequalities describing the region associated to a colored threshold graph are as follows:
$x_i + x_j > 0$ if and only if there is an edge between $i$ and $j$, $-\frac{1}{2} < x_i < \frac{1}{2}$ if $i$ is not colored, $x_i > \frac{1}{2}$ if $i$ is colored blue and $x_i<-\frac{1}{2}$ if $i$ is colored red.
Notice that the underlying labeled threshold graph corresponds to the $\mathcal{T}_n$ region that the $\mathcal{BT}_n$ region lies in.

Also, we can see that the bounded regions of $\mathcal{BT}_n$ are in bijection with the regions of $\mathcal{T}_n$.
Both are represented by labeled threshold graphs with $n$ vertices.
The bounded region of $\mathcal{BT}_n$ corresponding to a region of $\mathcal{T}_n$ is the one satisfying the same inequalities between $x_i+x_j$ and $0$ for all $i\neq j$ in $[n]$ and having $-\frac{1}{2} < x_i < \frac{1}{2}$ for all $i \in [n]$.

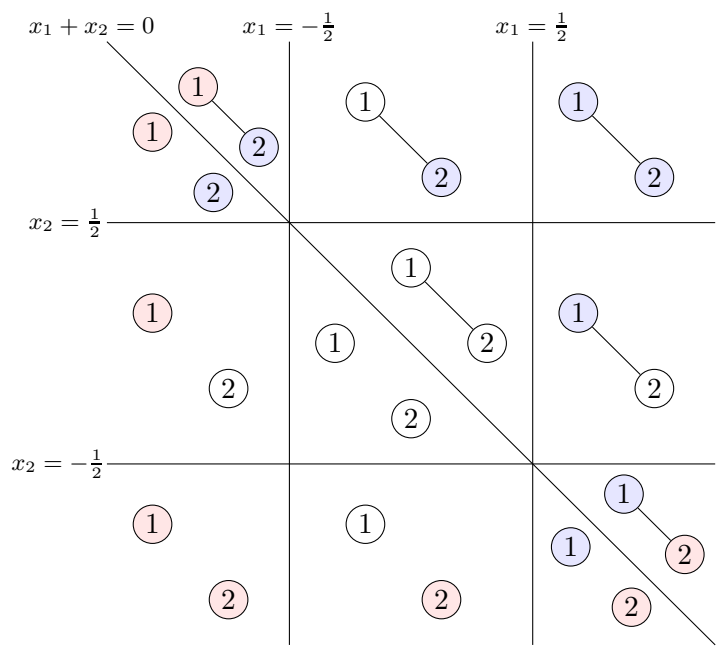
\begin{figure}[H]
    \centering
    \begin{tikzpicture}[scale=0.4]
    \draw (-4,10.5) node {\scriptsize $x_1=-\frac{1}{2}$};
    \draw (4,10.5) node {\scriptsize $x_1=\frac{1}{2}$};
    \draw (-11.6,-4) node {\scriptsize $x_2=-\frac{1}{2}$};
    \draw (-11.35,4) node {\scriptsize $x_2=\frac{1}{2}$};
    \draw (-10.5,10.5) node {\scriptsize $x_1+x_2=0$};
        \draw (-4,10)--(-4,-10);
        \draw (4,10)--(4,-10);
        \draw (-10,4)--(10,4);
        \draw (-10,-4)--(10,-4);
        \draw (-10,10)--(10,-10);
        \node (1) [circle,draw=black,inner sep=2pt] at (0,2.5) {\small 1};
        \node (2) [circle,draw=black,inner sep=2pt] at (2.5,0) {\small 2};
        \draw (1)--(2);
        \node  [circle,draw=black,inner sep=2pt] at (0,-2.5) {\small 2};
        \node  [circle,draw=black,inner sep=2pt] at (-2.5,0) {\small 1};
        
        \node  [circle,draw=black,fill=blue!10,inner sep=2pt] at (-6.5,5) {\small 2};
        \node  [circle,draw=black,fill=red!10,inner sep=2pt] at (-8.5,7) {\small 1};
        
        \node  [circle,draw=black,inner sep=2pt] at (-6,-1.5) {\small 2};
        \node  [circle,draw=black,fill=red!10,inner sep=2pt] at (-8.5,1) {\small 1};
        
        \node  [circle,draw=black,fill=red!10,inner sep=2pt] at (-6,-8.5) {\small 2};
        \node  [circle,draw=black,fill=red!10,inner sep=2pt] at (-8.5,-6) {\small 1};
        
        \node  [circle,draw=black,fill=red!10,inner sep=2pt] at (1,-8.5) {\small 2};
        \node  [circle,draw=black,inner sep=2pt] at (-1.5,-6) {\small 1};
        
        \node  [circle,draw=black,fill=red!10,inner sep=2pt] at (-6.5+13.75,-7-1.75) {\small 2};
        \node  [circle,draw=black,fill=blue!10,inner sep=2pt] at (-8.5+13.75,-5-1.75) {\small 1};
        
        \node (1') [circle,draw=black,fill=red!10,inner sep=2pt] at (-6.5+13.75+1.75,-7) {\small 2};
        \node (2') [circle,draw=black,fill=blue!10,inner sep=2pt] at (-8.5+13.75+1.75,-5) {\small 1};
        \draw (1')--(2');
        
        \node (1'')[circle,draw=black,inner sep=2pt] at (-6+14,-1.5) {\small 2};
        \node (2'')[circle,draw=black,fill=blue!10,inner sep=2pt] at (-8.5+14,1) {\small 1};
        \draw (1'')--(2'');
        
        \node (a)[circle,draw=black,fill=blue!10,inner sep=2pt] at (-6+14,-1.5+7) {\small 2};
        \node (b)[circle,draw=black,fill=blue!10,inner sep=2pt] at (-8.5+14,1+7) {\small 1};
        \draw (a)--(b);
        
        \node (a')[circle,draw=black,fill=blue!10,inner sep=2pt] at (-6+7,-1.5+7) {\small 2};
        \node (b')[circle,draw=black,inner sep=2pt] at (-8.5+7,1+7) {\small 1};
        \draw (a')--(b');
        
        \node (c) [circle,draw=black,fill=blue!10,inner sep=2pt] at (-6.5+1.5,5+1.5) {\small 2};
        \node (d) [circle,draw=black,fill=red!10,inner sep=2pt] at (-8.5+1.5,7+1.5) {\small 1};
        \draw (c)--(d);
    \end{tikzpicture}
    \caption{Regions of $\mathcal{BT}_2$ represented by labeled colored threshold graphs}
\end{figure}

\end{document}